\documentclass{nm}
\usepackage{charter}
\usepackage{amsmath}
\usepackage{mathtools}
\usepackage{setspace}
\usepackage[algo2e]{algorithm2e}
\usepackage{algpseudocode}
\usepackage{tikz}
\usepackage{graphicx}
\usepackage[colorlinks=true, allcolors=blue, hidelinks]{hyperref}

\usetikzlibrary{positioning,arrows.meta}
\usepackage{ulem}
\usepackage{subcaption}

\renewcommand\thisnumber{x}
\renewcommand\thisyear {2026}
\renewcommand\thismonth{x}
\renewcommand\thisvolume{xx}
\renewcommand\doi{10.4208/nmtma.  }

\begin{document}

\markboth{J. Kim and J. Xin}{An Efficient Particle-Field Algorithm based on a PHKS in 3D}
\title{An Efficient Particle-Field Algorithm with 
Neural Interpolation based on a Parabolic-Hyperbolic
Chemotaxis System in 3D}

\author[Jongwon David Kim and Jack Xin]{Jongwon David Kim\affil{1}\comma\corrauth, Jack Xin \affil{1}}
\address{
\affilnum{1} Department of Mathematics, University of California, Irvine, Irvine 92697, USA \  \\
[2ex]
\rm Received   ; Accepted (in revised version)  }

\emails{
{\tt jongwodk@uci.edu  } (Jongwon D. Kim),
{\tt jack.xin@uci.edu } (Jack Xin)
}

\begin{abstract}
Tumor angiogenesis involves a collection of tumor cells moving towards blood vessels for nutrients to grow. Angiogenesis, and in general chemotaxis systems have been modeled using partial differential equations (PDEs) and as such require numerical methods to approximate their solutions in 3 space dimensions (3D). This is an expensive computation when solutions develop large gradients at unknown locations, and so efficient algorithms to 
capture the main dynamical behavior are valuable. Here as a case study,  we consider a parabolic-hyperbolic Keller-Segel (PHKS) system in the angiogenesis literature, and develop a mesh-free particle-based neural network algorithm that scales better to 3D than traditional mesh based solvers. 
From a regularized approximation of PHKS, we derive a neural stochastic interacting particle-field (NSIPF) algorithm where the bacterial density is represented as empirical measures of particles and the field variable (concentration of chemo-attractant) by a convolutional neural network (CNN) trained on low cost 
synthetic data.  As a new model, NSIPF preserves total mass and non-negativity of the density, and captures the dynamics of 3D multi-bump solutions at much faster speeds compared with classical finite difference (FD) and spline based SIPF methods.

\end{abstract}



\maketitle

\section{Introduction}
Mathematical models are powerful tools to study a wide range of physical and biological phenomena. The field of cancer modeling includes various approaches from mechanistic models that explore the detailed biochemical mechanisms of diseases to data-driven models that facilitate clinical decision-making \cite{Bekisz2020}. Angiogenesis is the biological process of the formation of new blood vessels and provides a way for tumors to metastasize. This phenomenon has been widely studied by both clinical and computational scientists. Keller and Segel (KS) first introduced the chemotaxis system of partial differential equations (PDEs) to model the movement of bacteria to a food source (chemoattractant) \cite{KS70}. Since then, many PDE models of KS type have been used to study tumor growth models and their associated biological processes.\par 

In this paper, in the spirit of \cite{Wang2023,Hu2024} for fully parabolic chemotaxis and related haptotaxis systems, we introduce a stochastic interacting particle-field algorithm 
with neural interpolation (Neural SIPF) for a parabolic-hyperbolic KS (PHKS) system motivated by angiogenesis \cite{Corrias2003}.
Our method takes into account the coupled stochastic particle and field dynamics, where the field is the chemo-attractant concentration. In our Neural SIPF (NSIPF)  algorithm, we approximate the density of active particles by a sum  of delta functions centered at the particle positions. The non-smoothness of particle representation goes into the field due to hyperbolicity.  
An interpolation is necessary from particle representation to compute the gradient of the field that drives the particle evolution. While classical interpolators such as spline interpolation can be used for this step, we study neural interpolators due to their advantages in high dimensions. Neural interpolators have been of interest in scientific disciplines such as weather forecasting or earth sciences \cite{Leionen24, Zhan2023} and have also been of broader mathematical interest \cite{Wangetal2023}. In these works, training usually consists of millions of sample data samples obtained by experimentalists. 

In this paper, however, the neural network is trained on synthetic data produced by solving the radial PHKS system. Computing such a system is easier than solving the original PDE system by traditional methods, thus providing an easier route to generating solutions as training data. Our method is mesh free, easy to implement, and able to capture the diffusive and aggregation behavior of  the system as shown by comparison with finite difference method (FDM) and classical SIPF (e.g. with a spline interpolator).
In summary, our main contributions are as follows: 
\begin{itemize}
    \item We introduce a convolutional neural network architecture (CNN) to successfully interpolate SIPF solutions to the PHKS system in 2D and 3D. 
    \item The CNN interpolator is trained on radial solution data cheaply, thus removing the need of a fully resolved FDM for collecting training data.
    \item SIPF with CNN interpolation (NSIPF) is comparable to FDM and SIPF with classical interpolation (i.e. splines) in the quality of solutions, yet is easier to implement at lower computational costs. 
    \item Combining the above, we introduce a paradigm of simulating a regularized version of the original system by a particle-field-neural-network approach. It 
generates a new model that preserves physical properties (mass conservation and non-negativity) and 
is easier to scale up in multiple dimensions while maintaining qualitative properties of numerical solutions.
\end{itemize}

The rest of the paper is organized as follows: In section 2, we review the PHKS system of equations analyzed in \cite{Corrias2003}, and present a propagation of chaos theory for a convolution regularized PHKS system as a motivation for neural interpolation. In section 3, we describe the main SIPF algorithm that utilizes the theoretical stochastic differential equation (SDE) formulation of particle density as well as the need of a numerical interpolator to evolve the gradient of concentration field and particle positions.  In section 4, we show numerical results and convergence tests of the NSIPF algorithm on locally Gaussian shaped initial conditions for density and concentration. At higher resolutions in 3D, NSIPF has much smaller run times than FDM and classical SIPF while capturing major dynamical properties. Concluding remarks are in section 5.

\section{Parabolic-Hyperbolic KS (PHKS)}
In this section we present the PHKS system and discuss previous theoretical and numerical work. 
The original system is given below: 
\begin{align}
    \rho_t &= \nabla \cdot (\gamma \nabla \rho - \chi \rho \nabla c)\\
    c_t &= -c\, \rho
\end{align}
where \(\rho\) is the density of the bacteria and \(c\) is the concentration of the chemoattractant. The bacteria diffuse with mobility \(\gamma\) and drift in the direction of \(\nabla c\) with velocity \(\chi \nabla c\), where \(\chi\) is called chemo-sensitivity. In the context of tumor growth, this is known as a type of \textit{angiogenesis}, in which the endothelial cells of a tumor produce its own chemical to induce blood vessel growth. Theoretical analysis of well-posedness of this system is in \cite{Corrias2003} with the existence of a family of self-similar solutions for the 2D case. While this proves to be useful in training a neural interpolator in 2D, it appears unknown that such a family exists in 3D. 

\subsection{Propagation of chaos for the PHKS system}
The \textit{propagation of chaos} is a mathematical property of a system of \(N\) interacting particles that describes the limiting behavior of the system as the number of particles grows very large. Physically, this idea describes that for large systems the correlations between given particles due to interactions become negligible. In other words, the particles reach a ``statistical independence" as the system becomes larger, meaning that the \textit{average behavior} can be understood rather than the detailed correlated trajectories of each individual particle. The motivating scientific ideas root from Boltzmann's kinetic theory of gases in the 19th century but were later more mathematically formalized by Kac and McKean in the middle of the 20th century. \par 
In this subsection, we discuss the relevance of the propagation of chaos for a McKean-Vlasov interacting particle system and thus theoretical motivation for NSIPF algorithm. We refer to \cite{Diez2022} for further details and references therein. Propagation of chaos arguments are used to show that the solutions to interacting particle systems converge to solutions of the corresponding Fokker-Planck equation. In our case, we have the McKean-Vlasov process: 
\begin{equation}
    dX_t^i = \frac{1}{N} \sum_{j = 1}^N \mathcal{K}(X_t^i, X_t^j)dt + dB_t^i
\end{equation}
where the sum is a sum of binary interaction forces given by \(\mathcal{K}\) of particles \(X^i, X^j\) at time \(t\), in the overdamped regime, meaning there are no mechanical oscillations. \(B_t^i\) is the regular Brownian motion. The summation term above is given by 
\begin{equation}
    \frac{1}{N} \sum_{j = 1}^N\mathcal{K}(X_t^i, X_t^j) = \int_{\mathbb{R}^d} \mathcal{K}(X_t^i, y)\hat{\mu}_t^N(dy) \eqqcolon b(X_t^i, \hat{\mu}_t^N)
\end{equation}
where \(\hat{\mu}_t^N\) is the empirical measure \(\hat{\mu}_t^N = \frac{1}{N}\sum_{j = 1}^N \delta_{X_t^j} \in \mathcal{P}(\mathbb{R}^d)\), meaning that 
\[
b: \mathbb{R}^d \times \mathcal{P}(\mathbb{R}^d) \to \mathbb{R}^d
\]
and \(b\) is of the form \(b(x, \mu) \coloneqq \int \mathcal{K}(x,y)\mu(dy)\), which is linear in \(\mu\). In general, we have the following: 
\begin{equation}
    dX_t^i = b(X_t^i, \hat{\mu}_t^N)dt + dB_t^i 
\end{equation}
which we call the ``pointwise" McKean-Vlasov system. Since NSIPF depends on the previous time step, we must formulate it in the path-dependent sense. We define the probability measure associated to the law of the process: 
\begin{equation}
    \hat{\mu}_{[0, T]}^N \coloneqq \frac{1}{N}\sum_{i = 1}^N \delta_{X_{[0, T]}^i} \in \mathcal{P}(\mathcal{C}_T).
\end{equation}
Now the pathwise McKean-Vlasov system is written as 
\begin{equation}
    dX_t^i = \tilde{b}(X_{[0, T]}^i, \tilde{\mu}_{[0, T]}^N)dt + dB_t^i
\end{equation}
where \(\tilde{b}: \mathcal{C}_T \times \mathcal{P}(\mathcal{C}_T) \to \mathbb{R}^d\). In the case of PHKS, we have
\begin{equation}
    c(t) = \exp\bigg(-\int_0^t \rho(s)ds\bigg)\, c_0(x).
\end{equation}
To connect the previous discussion to our equation, we define the function \(b\) to be 
\begin{equation}
    b(x, \rho) = \nabla_x \bigg\{\exp\bigg(-\int_0^t \rho(s)ds\bigg)c_0(x)\bigg\}.
\end{equation}
However, \(b: \mathbb{R}^d \times \mathcal{P}(\mathcal{C}_T) \to \mathbb{R}^d\), but \(\rho\) is not just a measure but a probability \textit{density} and therefore belongs in \(\mathcal{P}^{\text{ac}}(\mathcal{C}_T)\). The empirical measure \(\hat{\mu}^N\) is not absolutely continuous but a \textit{discrete} measure. As such, we cannot simply define \(b\) as above. To address this, we smoothe it with a mollifier. Choose \(\delta > 0\) and take \(\mathcal{K}_\epsilon \to \delta_0\) as \(\epsilon \to 0\). The equation for \(c\) should then be 
\begin{equation}
    c_t = - c(\mathcal{K}_\epsilon \star \rho).
\end{equation}
Define  
\begin{equation}
    b_\epsilon(x, \rho) := \exp\bigg(-\int_0^t \mathcal{K}_\epsilon \star \rho(s)ds\bigg)\bigg(\nabla c_0 - c_0(x) \int_0^t \nabla \mathcal{K}_\epsilon \star \rho(s)ds\bigg).   
\end{equation}
Now, \(b_\epsilon\) is defined appropriately with \(\mathcal{P}(\mathcal{C}_T)\) and the particle system is now well-defined. If \(b_\epsilon\) is Lipschitz bounded, then we are able to utilize a propagation of chaos statement. In particular, we refer to the concept of \textit{moderate interaction} introduced by Oelschläger in the 1980's for systems with constant diffusion matrix \(\sigma \equiv \sqrt{2}I_d\) and with a symmetric interaction kernel \(K_1\) which \textit{depends on \(N\)}:
\begin{equation}
    \forall x, y, \in \mathbb{R}^d, \hspace{4mm} K_1(x,y) \equiv K_1^N(y - x) \coloneqq \frac{1}{\epsilon_N^d}K_0\bigg(\frac{y - x}{\epsilon_N}\bigg),
\end{equation}
where \(K_0: \mathbb{R}^d \to \mathbb{R}\) is a fixed symmetric radial kernel and \((\epsilon_N)_N\) is a sequence such that \(\epsilon_N \to 0\) as \(N \to +\infty\). The strength of the interaction between two particles is of the order \(\sim \epsilon_N^{-d}N^{-1}\). Oelschläger considered the case \(\epsilon_N = N^{-\beta/d}\) with \(\beta \in (0, 1)\). Note that if \(\beta = 0\) then there's weak interaction (\(\mathcal{O}(N^{-1})\)) and \(\beta = 1\) gives strong interaction (\(\mathcal{O}(1)\)), hence the term \textit{moderate interaction}. We refer this to any situation where  \(\epsilon_N \to 0\) and \(\epsilon_N^{-d}N^{-1} = o(1)\), which leads to 
\[
K_1^N(x, \cdot) \xrightarrow[N \to +\infty]{} \delta_x,
\]
in the sense of distribution, allowing us to recover singular \textit{purely local} interactions. In the NSIPF algorithm, we view the convolutional neural network framework as a proxy convolution operation. To establish a mathematical statement, we introduce the following assumption:

\begin{assumption}
For a fixed \(\epsilon > 0\) and probability measures \(\mu, \nu\), if \(|b_\epsilon(x, \mu) - b_\epsilon(y, \nu)| \leq C(|x - y| + W_2(\mu, \nu))\) (where \(W_2\) is the Wasserstein-2 distance), then for all finite \(T > 0\), the Fokker-Planck equation is well-posed in the pathwise sense and the associated SDE has a unique strong solution.
\end{assumption} With this, we present the following theorem: 

\begin{theorem}{(Propagation of chaos with moderate interaction)}
Let the interacting particle system and its mollification for PHKS system be 
\begin{equation}
    dX_t^i = \chi b(X_t^i, f_t)dt + \sqrt{2\gamma}dB_t^i
\end{equation}
\begin{equation}
    d\bar{X}_t^{i, N} = \chi b_\epsilon(\bar{X}_t^{i, N}, K_\epsilon \star \mu_{[0, T]}^N)dt + \sqrt{2\gamma}dB_t^i
\end{equation}
where \(b_\epsilon\) is defined by (2.11). Then, by synchronous coupling with moderate interaction, the propagation of chaos holds for the PHKS system. As \(N \to \infty\), the empirical measure of the particle system (2.14) converges weakly to the measure with density $\rho$ satisfying the regularized PHKS system: 
\begin{align}
    \rho_t &= \nabla \cdot (\gamma \nabla (\mathcal{K}_\epsilon \star \rho) - \chi (\mathcal{K}_\epsilon \star \rho) \nabla c)\\
    c_t &= - c(\mathcal{K}_\epsilon \star \rho).
\end{align}

\end{theorem}
\begin{proof}
In order to show that (2.14) converges to (2.13), we require limits as \(N \to +\infty\) and \(\epsilon \to 0\). The ideas of the proof follow closely Sznitman's proof of McKean's theorem of Lipschitz interactions using synchronous coupling which can be found in the review paper \cite{Diez2022} and moderate interaction which can be found in \cite{JourdainMeleard1998}. It first suffices to show that \(\lim_{N \to +\infty} \mathbb{E}[\sup_{t \leq T} |X_t^i - \bar{X}_t^i|^2] = 0\). Let \(\chi = 1\), \(\mu = \mu_{[0, T]}^N = \mu_{X_t^N}\), \(\bar{\mu} = \mu_{\bar{X}_t^N} = \frac{1}{N}\sum_{i = 1}^N\delta_{\bar{X}_t^i}\) and \(f_t\) be the law of the associated system. By the Burkholder-Davis-Gundy (BDG) inequality, for \(i \in \{1, 2, ..., N\}\), 
\begin{align}
&    \mathbb{E}[\sup_{t \leq T} |X_t^i - \bar{X}_t^i|^2] \leq 2T\int_0^T \mathbb{E}|b_\epsilon(X_t^i, \mu) - b(\bar{X}_t^i, f_t)|^2dt\\
    &\leq 4T\int_0^T\mathbb{E}|b_\epsilon(X_t^i, \mu) - b_\epsilon(\bar{X}_t^i, \bar{\mu})|^2 + \mathbb{E}|b_\epsilon(\bar{X}_t^i, \bar{\mu}) - b_\epsilon(\bar{X}_t^i, f_t)|^2dt
\end{align}
Then, 
\begin{align}
    \mathbb{E}|b_\epsilon(X_t^i, \mu) - b_\epsilon(\bar{X}_t^i, \bar{\mu})|^2 &\leq C(\mathbb{E}|X_t^i - \bar{X}_t^i|^2 + \mathbb{E}W_2^2(\mu, \bar{\mu}))\\
    &\leq C(\mathbb{E}|X_t^i - \bar{X}_t^i|^2 + \frac{1}{N}\sum_{j = 1}^N\mathbb{E}|X_t^j - \bar{X}_t^j|^2)\\
    &\leq 2C\mathbb{E}|X_t^i - \bar{X}_t^i|^2
\end{align}
and \(\mathbb{E}|b_\epsilon(\bar{X}_t^i, \bar{\mu}) - b_\epsilon(\bar{X}_t^i, f_t)|^2 \leq C\mathbb{E}W_2^2(\bar{\mu}, f_t)\). Combining, we have 
\begin{align}
    \mathbb{E}[\sup_{t \leq T}|X_t^i - \bar{X}_t^i|^2] &\leq C_1\int_0^T \mathbb{E}W_2^2(\bar{\mu}, f_t)dt + C_2 \int_0^T \mathbb{E}|X_t^i - \bar{X}_t^i|^2dt\\
    &\leq C_1e^{C_2T}\int_0^T\mathbb{E}W_2^2(\bar{\mu}, f_t)dt\\
    &\leq C_1e^{C_2T}\frac{C_3}{N},
\end{align}
where the second to last inequality follows from Gronwall's inequality. Then, to incorporate the moderate interaction, let \(\epsilon = \epsilon_N\) be a sequence of numbers converging to 0 and \(K_\epsilon = \epsilon_N^{-d}K(\frac{x}{\epsilon_N})\) where \(K\) is a Lipschitz continuous and bounded probability density on \(\mathbb{R}^d\). We first note that the \(L^\infty\) and Lipschitz norms of \(K_\epsilon\) are controlled by \(\|K_\epsilon\|_\infty = \frac{C_0}{\epsilon_N^d}\) and \(\|K_\epsilon\|_{\text{Lip}} = \frac{C_1}{\epsilon_N^{d+1}}\) for some constants \(C_0, C_1 > 0\) depending on \(K\). Then, McKean's theorem gives that for all \(N\), 
\begin{equation}
    \mathbb{E}[\sup_{t \leq T} |X_t^{i, N} - \bar{X}_t^{i, N}|^2] \leq \tilde{c}_1\frac{\epsilon_N^{-2d}}{N}\exp(\tilde{c}_2\epsilon_N^{-2(d+1)}),
\end{equation}
for some constants \(\tilde{c}_1, \tilde{c}_2 > 0\) depending on \(T, K,\) and \(b_\epsilon\). In order to take \(N \to +\infty\), \cite{JourdainMeleard1998} assume that \(\epsilon_N \to 0\) slowly enough such that the right hand side of (2.22) converges to 0. A sufficient condition is \(\epsilon_N^{-2(d + 1)} \leq \delta \log N\) for a small \(\delta > 0\). \cite{JourdainMeleard1998} further show that assuming that the SDE is well-posed, then
\begin{equation}
    \mathbb{E}[\sup_{t \leq T}|\bar{X}_t^{i, N} - \bar{X}_t^i|^2] \leq C\epsilon_N^{2\beta}
\end{equation} for some \(\beta > 0\). Combining these two results, we have then that 
\begin{equation}
    \mathbb{E}[\sup_{t \leq T}|X_t^{i} - \bar{X}_t|^2] \leq C\epsilon_N^{2\beta} + \tilde{c}_1\frac{\epsilon_N^{-2d}}{N}\exp(\tilde{c}_2\epsilon_N^{-2(d+1)}),
\end{equation} and the conclusion follows.

\end{proof}

\begin{remark}\label{rem1}
We remark that the proof of propagation of chaos depends on the regularity of \(c_0\) which is not smooth in general. In the case that it is, or at least \(C^2\) differentiable, the explicit Lipschitz constants can be formulated and the overall argument follows the same way.
\end{remark}

\subsection{Towards a neural approach}
Due to the propagation of chaos through mollification and moderate interaction, we are inspired to incorporate neural networks to satisfy the role of \(b_\epsilon\) due to the universal approximation capability of neural networks. In this case, we seek to define a \textit{data-driven} operator \(\mathcal{M}_\theta[f](\textbf{x})\) where \(f\) is the function to be interpolated, \(\textbf{x} \in \mathbb{R}^d\) is the data input, and \(\theta\) is the trained parameters by the neural network. Due to the role of convolution in the propagation of chaos theory, we consider using a convolutional neural network (CNN). Here, a neural network interpolator \(\mathcal{M}_\theta\) defined for \(L\) layers, is 
\begin{align}
    \mathcal{M}_\theta[f](x) &= h^{(L)}\\
    h^{(l)} &= \sigma (K^{(l)} \star h^{(l-1)} + b^{(l)}), \hspace{4mm} l = 1, ..., L\\
    h^{(0)} &= f
\end{align}
where \(K^{(l)}\) are convolutional kernels, \(\sigma\) is our activation function, and \(f\) is the input as a collection of data points \(\{f(x_i)\}_i\) where \(x_i\) are the points where \(f\) is defined (say on a uniform grid). A classical interpolator \(\mathcal{I}\) typically asserts
\begin{equation}
    \mathcal{I}[f](x_i) = f(x_i)
\end{equation}
for given data \(\{x_i\}\). However, the neural interpolator is an \textit{inexact} method, meaning that \(\mathcal{M}[f](x_i) \approx f(x_i)\).

\section{Computational methods for PHKS}
In this section, we present the main algorithms for simulating solutions to the PHKS system, including the Neural SIPF algorithm as well as the training procedure of the neural interpolator. Instead of using a classical interpolation method like cubic splines, we opt for a CNN trained on radial solutions. 

\subsection{Neural SIPF algorithm}
Consider a finite spatial domain \(\Omega = [0, L]^d\) with Neumann boundary conditions for \(\rho\) and \(c\). As a discrete time algorithm, we partition the time interval \([0, T]\) by \(\{t_n\}_{0:n_T}\) where \(t_0 = 0\) and \(t_{n_T} = T\). Approximate  density \(\rho\) by particles as:
\begin{equation} \label{rho_update}
\rho_t \approx \frac{M_0}{P}\sum_{p = 1}^P \delta(x - X_t^p),
\end{equation}
a.k.a the empirical measure, where \(P \gg 1\) is the number of particles and \(M_0\) is the conserved mass of the system. At \(t_0 = 0\), we sample \(P\) particles from the initial condition \(\rho_0\). To present the algorithm, we rewrite the particle approximation by \(\rho_n = \frac{M_0}{P}\sum_{p = 1}^P \delta(x - X_n^p)\). At a given time step, our algorithm consists of two sub-steps: updating \(c\) and \(\rho\). 
\medskip

\noindent 
\textit{Updating chemical concentration \(c\)}: Let \(\delta t = t_{n + 1} - t_n > 0\) be the time step. We update \(c\) by the explicit Euler scheme: 
\begin{align}
c_{n+1} = c_{n} - \delta t\, c_{n}\, \rho_{n}.
\end{align}
\textit{Updating density \(\rho\)}: 
update the particle positions \(\{X_n^p\}_{p = 1:P}\) using an Euler-Maruyama scheme of the SDE: 
\begin{equation} \label{particle_update}
X_{n+1}^p = X_n^p + \chi \, \nabla_{x}c(X_n^p, t_n)\, \delta t + \sqrt{2\gamma\delta t}\, N_n^p
\end{equation}
where \(N_n^p\)'s are i.i.d. standard normal distributions with respect to Brownian paths in the SDE formulation. The particles are then binned into \(N_g\) bins. \par 
Each timestep requires \(\mathcal{O}(P)\) operations to assign particles to bins and \(\mathcal{O}(N_g)\) operations to create the bins, making the overall complexity \(\mathcal{O}(P + N_g)\), rendering the method efficient for high-dimensional problems. \par 
Computing \(\nabla_x c(x, t_n)\) is difficult 
as it may not be strictly defined 
due to $\rho_{n-1}$ being only coarsely valued through a histogram of particle positions. To circumvent this, we first construct an everywhere differentiable interpolation of $c_n$ as a proxy, then take gradient on the interpolator for an approximation. A similar treatment has been adopted in training neural networks with piecewise constant activation functions. The resulting 
proxy gradient is known as coarse gradient \cite{yin2018blended,yin2018understanding,long2023recurrence} or straight-through-estimator \cite{Hubara2017QuantizedNN}. 

A classical interpolation is spline which is expensive to generalize to 3D and will be compared with neural interpolation in Alg. \ref{alg:one}.

\RestyleAlgo{ruled}
\normalem

\begin{algorithm2e}
    \caption{Neural SIPF}\label{alg:one}
    \DontPrintSemicolon
    \SetKwInOut{Input}{input}\SetKwInOut{Output}{output}
    \Input{\(\Omega, T, \chi, \mu, P, M_0, \delta t, \rho_0, c_0.\)}
    \Output{\(\rho_T\), \(c_T.\)}
    Initialize \(X^{1}_{0}, ..., X^{P}_{0}\) on \(\Omega\) based on initial data \(\rho_0.\)\;
    \For{\(n = 0, ..., n_T\)}
    {
        Bin particles \(X^{1}_{n}, ..., X^{P}_{n}\) and define \(\rho_n\) according to (\ref{rho_update})
        \(\rho_n \gets \text{histogram}(\rho_n)\)\;
        \uIf {\(n = 0\)}
        {
        \(c_n \gets c_0\)
        }
        \Else
        {\(c_{n+1} \gets c_{n} - \delta t\, c_{n}\, \rho_{n}\), over each bin\\
        \(c(x,t_n) \gets \text{CNN interpolator}(c_n)\)\\
        Compute \(\nabla_x c(x,t_n)\)\\
        Update \(X_{n}\) to \(X_{n+1}\) by (\ref{particle_update}) with \(\nabla_x c(x,t_n)\).
        }
    }
\end{algorithm2e}

\subsection{Neural interpolator training}
In the 2D case, a neural interpolator was trained on radially symmetric self-similar solutions computed by a finite difference method. These solutions were processed as 2D images by the neural network. However, processing 3D solutions on the whole spatial domain is computationally expensive and thus requires a different training strategy when implementing the neural interpolator to interpolate the concentration variable. \par
To train the neural network, we first use a finite difference method to compute solutions to the radial system below.
\begin{align}
    \rho_t &= \gamma (\rho_{rr} + \frac{2}{r}\rho_r) + \chi (\rho_rc_r + \rho c_{rr} + \frac{2}{r}\rho c_r)\\
    c_t &= -c\, \rho.
\end{align}
These solutions were computed with \(\gamma = \chi = 1\) with \(M_0 = 1\) on the space \([0, 100]\) with \([0, 50]\) for time.
Then, taking a random sample of 50 solutions across the time domain, we augment a patch of each 3D reconstructed solution to use as training data for the neural network interpolation. During augmentation, we perform a combination of down-sampling, shifting, and Gaussian blurring of the patch by randomly selecting three separate parameters within certain ranges during training. For down-sampling, shifting, and blurring respectively, the ranges were set to be \((2, 4), (-20, 20), (0, 0.5)\).  The patches are then resized to fit the dimensions of the input \(c\) defined on a uniform grid \(D \times H \times W\). Serving as the input of the neural interpolator, a patch is then interpolated to a defined resolution by passing through several convolutional layers that follow an encoder-decoder like structure. The details of the architecture including the changes of channel sizes are shown in Fig. \ref{cnndiagram}. We train the neural network for 100 epochs with Adam optimizer, learning rate of \(10^{-3}\), and a batch size of 4 with mean squared loss (MSE) as our loss function. The neural network training took 9753.73 seconds on an NVIDIA GTX GeForce 1080 GPU, with a 80/20 training/validation split. The loss vs. epoch in training can be found in Fig. \ref{trainingloss}. We remark that the loss during training does not go completely to zero, indicating 
that neural interpolator serves as an inexact interpolation method. Nonetheless, we shall see that the method helps SIPF capture the dynamic behavior in comparison with FDM solutions and at a much faster speed in 3D.

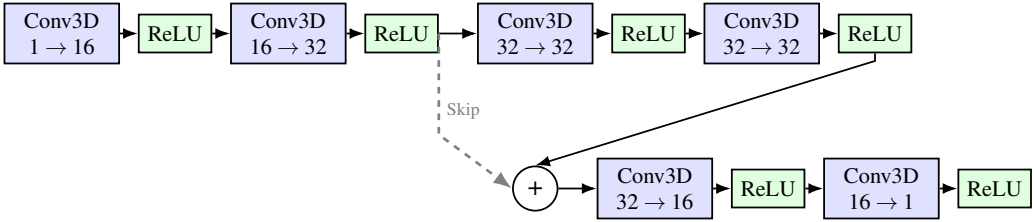
\begin{figure}
    \centering
    \scalebox{1.0}{
\resizebox{\textwidth}{!}{ 
\begin{tikzpicture}[
  conv/.style = {draw, thick, minimum width=1.8cm, minimum height=0.9cm,
                 fill=blue!12, align=center, font=\footnotesize},
  act/.style  = {draw, thick, minimum width=1.05cm, minimum height=0.6cm,
                 fill=green!12, align=center, font=\footnotesize},
  add/.style  = {draw, circle, thick, minimum size=7mm, inner sep=0pt, fill=white},
  arrow/.style = {-Latex, thick},
  skip/.style  = {-Latex, very thick, dashed, gray}
]

\node[conv] (conv1) {Conv3D\\$1 \rightarrow 16$};
\node[act, right=0.28cm of conv1] (relu1) {ReLU};
\node[conv, right=0.28cm of relu1] (conv2) {Conv3D\\$16 \rightarrow 32$};
\node[act, right=0.28cm of conv2] (relu2) {ReLU};

\node[conv, right=0.6cm of relu2] (conv3) {Conv3D\\$32 \rightarrow 32$};
\node[act, right=0.28cm of conv3] (relu3) {ReLU};
\node[conv, right=0.28cm of relu3] (conv4) {Conv3D\\$32 \rightarrow 32$};
\node[act, right=0.28cm of conv4] (relu4) {ReLU};

\node[add, below=1.6cm of conv3] (add) {+};
\node[conv, right=0.6cm of add] (conv5) {Conv3D\\$32 \rightarrow 16$};
\node[act, right=0.28cm of conv5] (relu5) {ReLU};
\node[conv, right=0.28cm of relu5] (conv6) {Conv3D\\$16 \rightarrow 1$};
\node[act, right=0.28cm of conv6] (relu6) {ReLU};

\draw[arrow] (conv1) -- (relu1);
\draw[arrow] (relu1) -- (conv2);
\draw[arrow] (conv2) -- (relu2);
\draw[arrow] (relu2) -- (conv3);
\draw[arrow] (conv3) -- (relu3);
\draw[arrow] (relu3) -- (conv4);
\draw[arrow] (conv4) -- (relu4);
\draw[arrow] (relu4.south) -- ++(0,-0.12) -- (add.north);
\draw[arrow] (add) -- (conv5);
\draw[arrow] (conv5) -- (relu5);
\draw[arrow] (relu5) -- (conv6);
\draw[arrow] (conv6) -- (relu6);

\coordinate (skip_down) at ([yshift=-1.6cm]relu2.east);
\draw[skip] (relu2.east) -- (skip_down) -- (add.west);

\node[font=\scriptsize, align=center, gray, above right=0.1cm and -0.02cm of skip_down] {Skip};

\end{tikzpicture}
}
}
    \caption{Diagram of the CNN architecture.}
    \label{cnndiagram}
\end{figure}

\begin{figure*}
    \centering
    \includegraphics[width=0.75\textwidth]{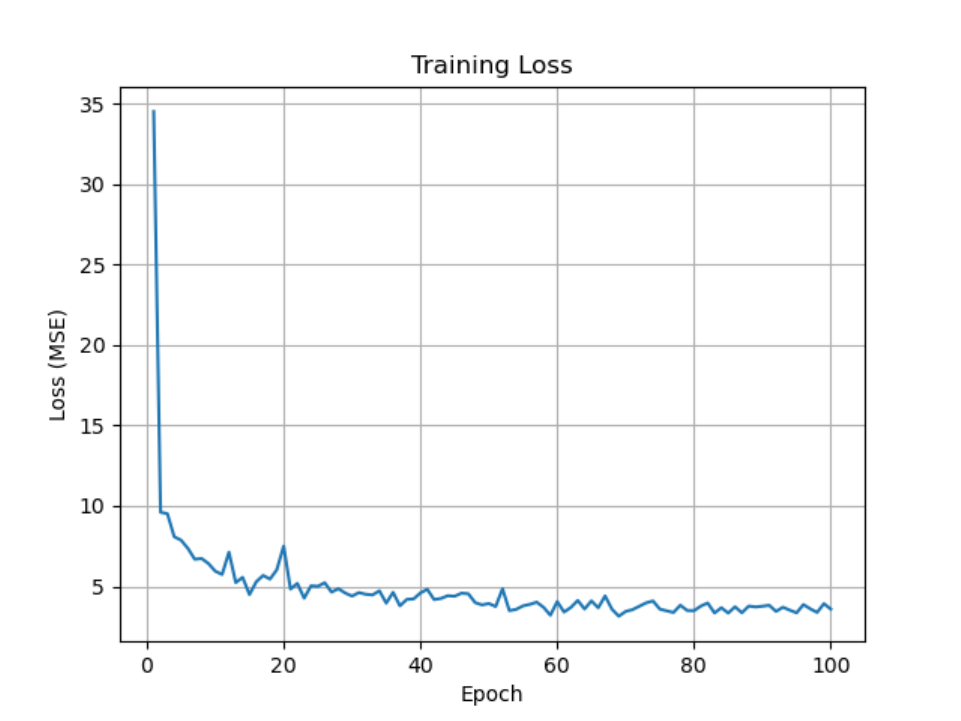}
    \caption{CNN training loss vs. epochs.}
    \label{trainingloss}
\end{figure*}

\section{Numerical Experiments and Discussion}
In this section we provide numerical results of the SIPF algorithm for one blob and two blob initial conditions as well as demonstrate the efficiency of the algorithm compared to classical methods such as classical SIPF or FDM. We provide numerical convergence results and discuss important points of the algorithm. For simplicity unless otherwise stated, all simulations initialize density and concentration as Gaussian blobs with \(\gamma = \chi = 1\) and assume \(M_0 = 1\) for density.  To implement the classical interpolator we use Python's RectBivariateSpline package for the 2D case and the interpn package for the 3D case. For FDM experiments, the spatial domain is discretized as a uniform grid and the differential operators in the PDE are discretized as forward differences. For classical SIPF experiments, the algorithm remains the same except the interpolation step is replaced with spline interpolation. All computational experiments were conducted on Python and training the neural network was done in PyTorch. 

\subsection{Diffusive behavior with one blob}
As a first experiment, we consider an initial standardized Gaussian density blob with standard deviation \(5.0\) positioned at the center of the spatial domain \((50, 50, 50)\) with an initial concentration of food source which is also a standardized Gaussian blob with standard deviation \(10.0\). The initial blob diffuses out as it consumes the food source. The blob diffuses outward as shown in Figure \ref{onebump}. As a demonstration that the neural interpolation is an inexact method, a cross-section plot is shown in Fig. \ref{crossection}.

\begin{figure*}[!tbh]
    \centering
    \begin{subfigure}[t]{0.48\textwidth}  
        \centering 
        \includegraphics[width=1.1\textwidth]{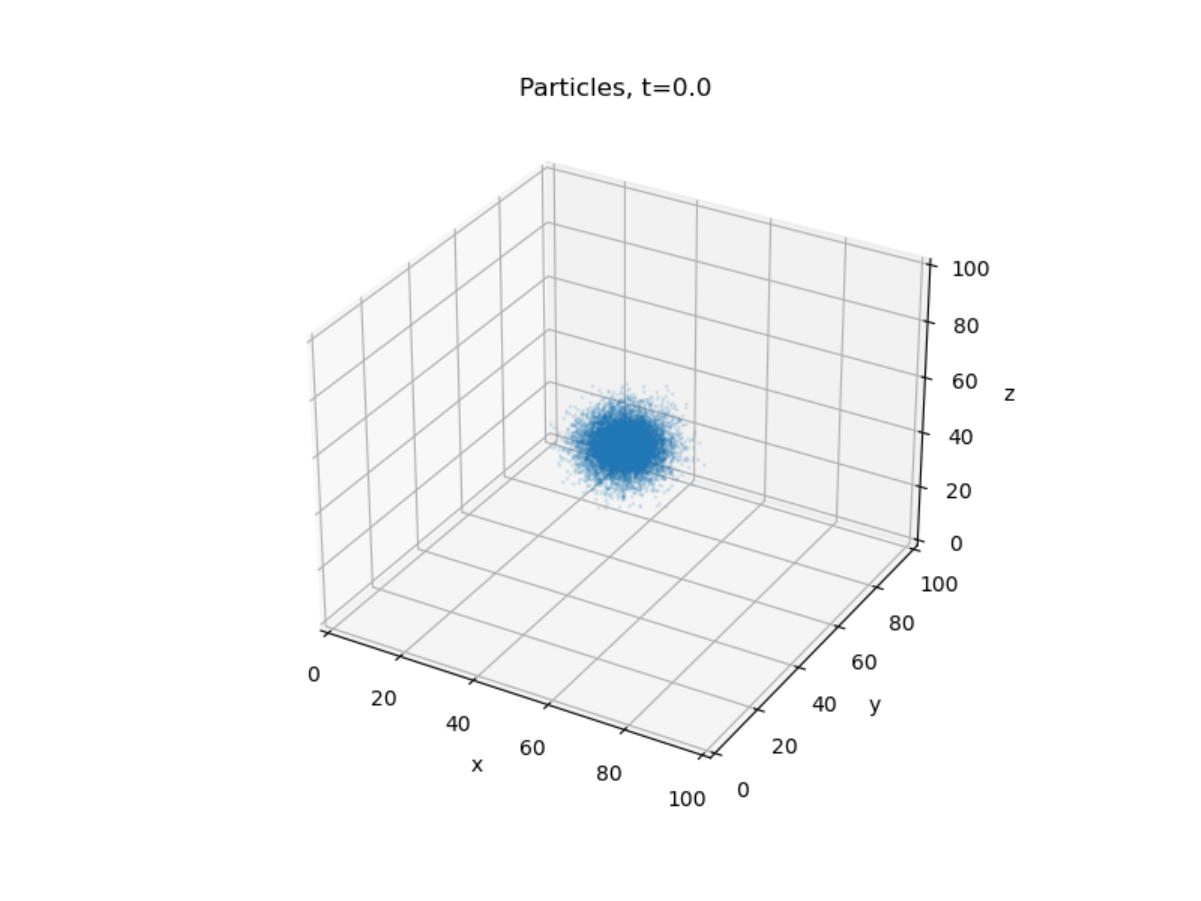}

    \end{subfigure}
    \hfill
    \begin{subfigure}[t]{0.48\textwidth}   
        \centering 
        \includegraphics[width=1.1\textwidth]{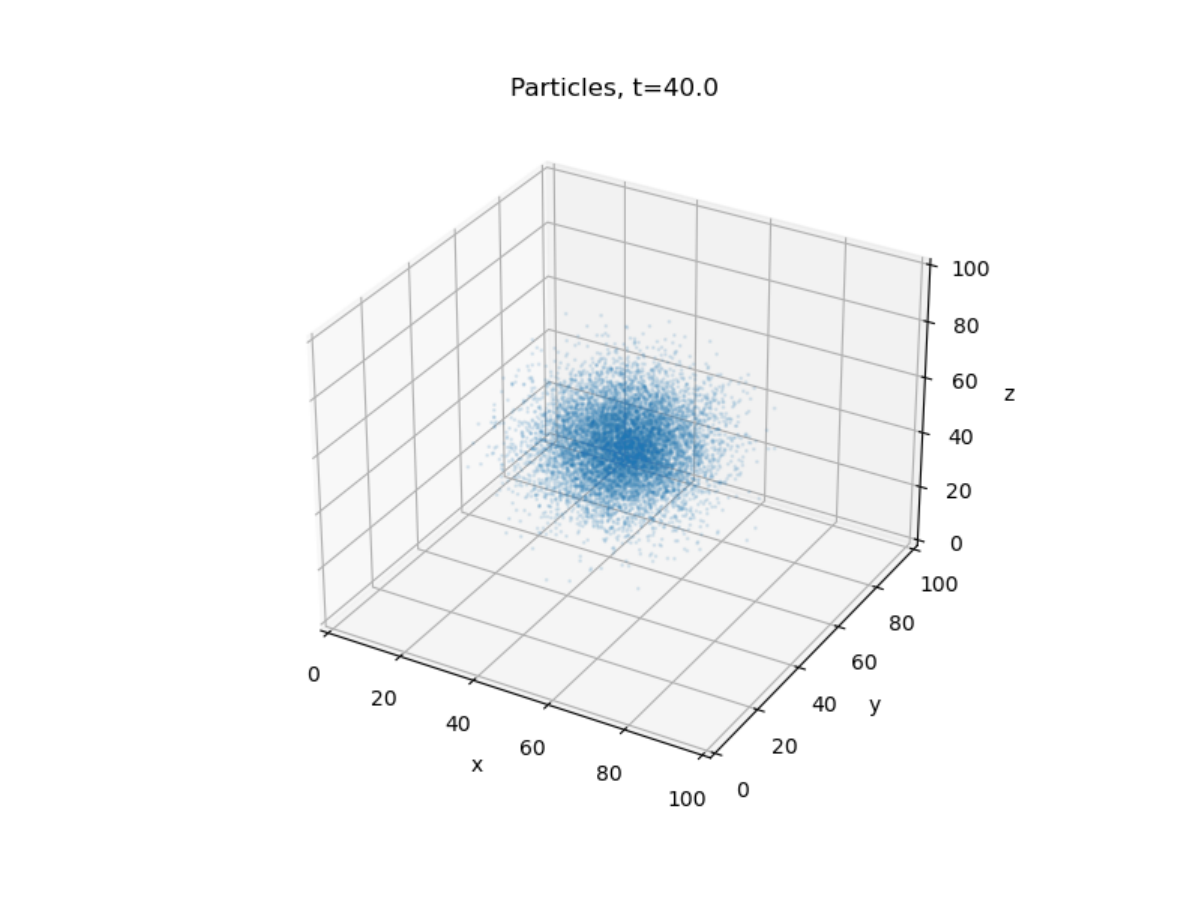}

    \end{subfigure}
    \caption[]%
    {{\small One bump solution particles at $t = 0$ and $t = 40$ produced by NSIPF with \(P = 20000\).}}
    \label{onebump}
\end{figure*}

\begin{figure*}
    \centering
    \includegraphics[width=0.75\linewidth]{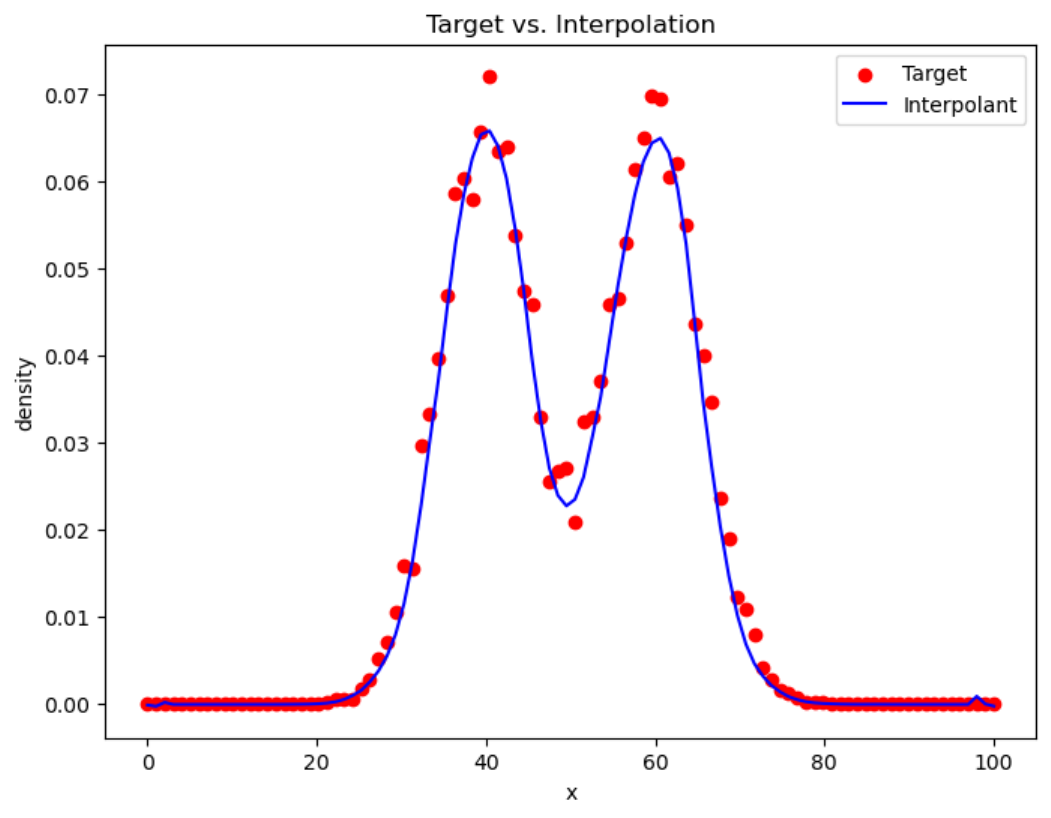}
    \caption{Target data points (red) and their neural interpolation at the \(y = 50\) slice.}
    \label{crossection}
\end{figure*}

\subsection{Aggregation behavior with two blobs}
In this experiment, we demonstrate the power of NSIPF to simulate solutions to the PHKS system with non radial initial data. We initialize \(\rho\) to be a sum of two Gaussian blobs of standard deviation \(5.0\) positioned off-center (\((30, 30, 30)\) and \((70, 70, 70)\)) in the spatial domain while \(c\) is a single Gaussian blob of standard deviation \(10.0\) located in the center (\((50, 50, 50)\)). The concentration is scaled to be much larger than density to provide more biological realism to show that the blobs aggregate toward the food source in finite time. As we can see in Fig. \ref{2bumpparticles}, the blobs aggregate towards the center where the concentration is and continues to grow there. Fig. \ref{2bump} shows cross-section plots of density comparing FDM and NSIPF with 20,000 particles.

\subsection{Aggregation behavior along annuli boundary}
In this last experiment, we demonstrate NSIPF's ability to simulate solutions to the PHKS system with different initial conditions for the concentration variable. In particular, we initialize concentration to be a sum of two annuli in 3D, which is different from the blob training data provided for the neural interpolator. Additionally, we vary \(\gamma, \chi, L, T\), and resolution outside of their ranges during the neural interpolator training while comparing to an FDM reference using 20,000 particles. The Wasserstein-1 error for \(\rho\) and relative error for \(c\) are computed at \(T\). The density is a single Gaussian blob located in the center of the domain. We see that in Fig. \ref{annuli}, the particles quickly aggregate along the boundary of the annuli where the food source is highest. The comparison with the FDM reference under different parameters is shown in Table \ref{table: annulicomparison}. Even when the parameters are changed, we observe that the error remains small for both \(\rho\) and \(c\), indicating that the NSIPF is still able to produce solutions comparable to traditional FDM with different initial conditions. Furthermore, there is no retraining required for a different set of parameters.

\begin{table}[h!]
    \centering
    \begin{tabular}{|c|c|c|}
    \hline
        \((\gamma, \chi, L, T)\) & Resolution & Error \((\rho, c)\) \\
        \hline
        \((1, 1, 100, 50)\) & \(200 \times 200 \times 200\) & (8.60e-07, 0.006)\\
         \hline
        \((10, 5, 100, 100)\) & \(200 \times 200 \times 200\) & (1.05e-06, 0.0004)\\
        \hline
        \((10, 5, 100, 5)\) & \(400 \times 400 \times 400\) & (9.76e-07, 0.001)\\
        \hline
        \((0.1, 0.1, 100, 50)\) & \(100 \times 100 \times 100\) & (6.89e-07, 0.005)\\
        \hline 
        \((1, 1, 100, 5)\) & \(250 \times 250 \times 250\) & (9.24e-07, 0.0008)\\
        \hline 
        \((1, 1, 200, 50)\) & \(250 \times 250 \times 250\) & (1.02e-07, 2.7e-08)\\
        \hline 
        
    \end{tabular}
    \caption{Comparison of FDM and NSIPF methods with annuli initial condition for \(c\)}
    \label{table: annulicomparison}
\end{table}

\begin{figure*}[t!]
    \centering
    \begin{subfigure}[t]{0.48\textwidth}  
        \includegraphics[width=1.1\textwidth]{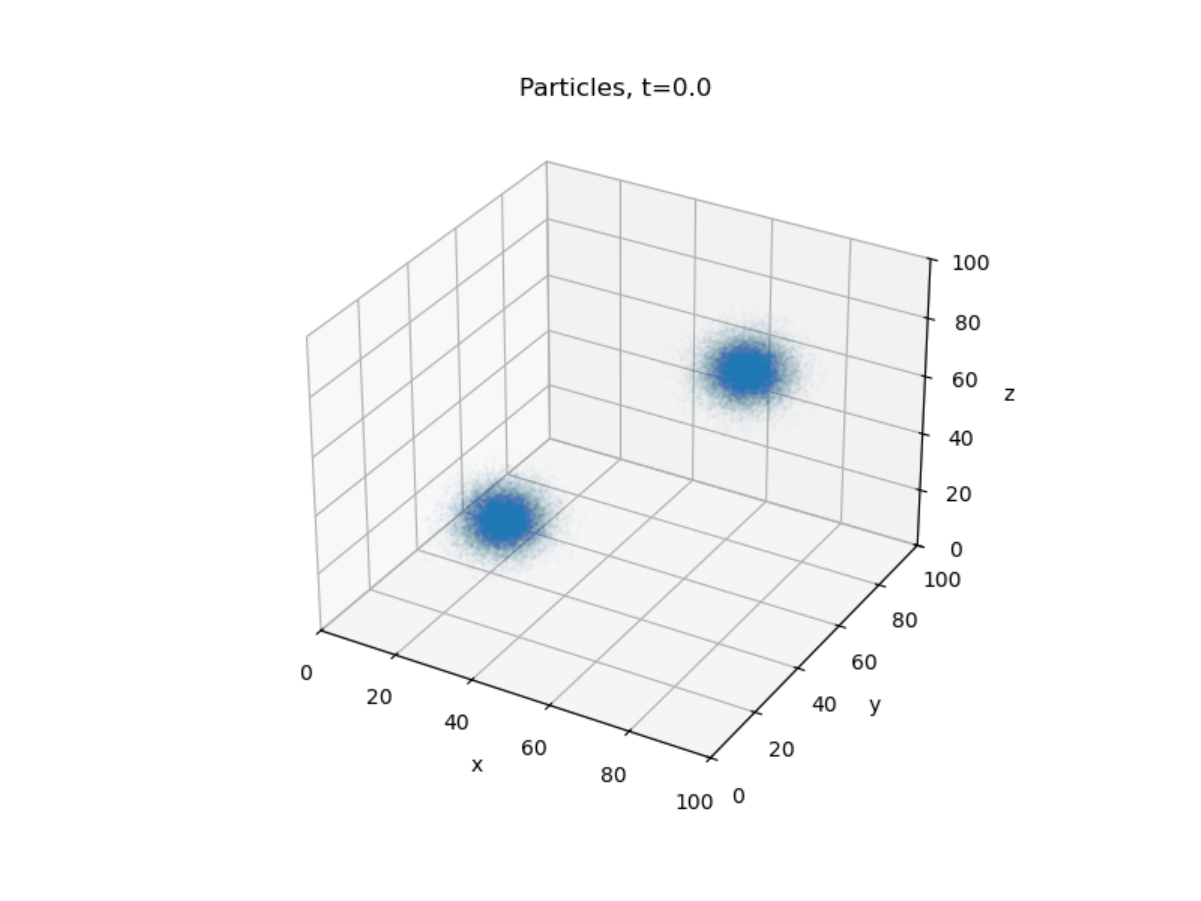}

    \end{subfigure}
    \begin{subfigure}[t]{0.48\textwidth}   
        \includegraphics[width=1.1\textwidth]{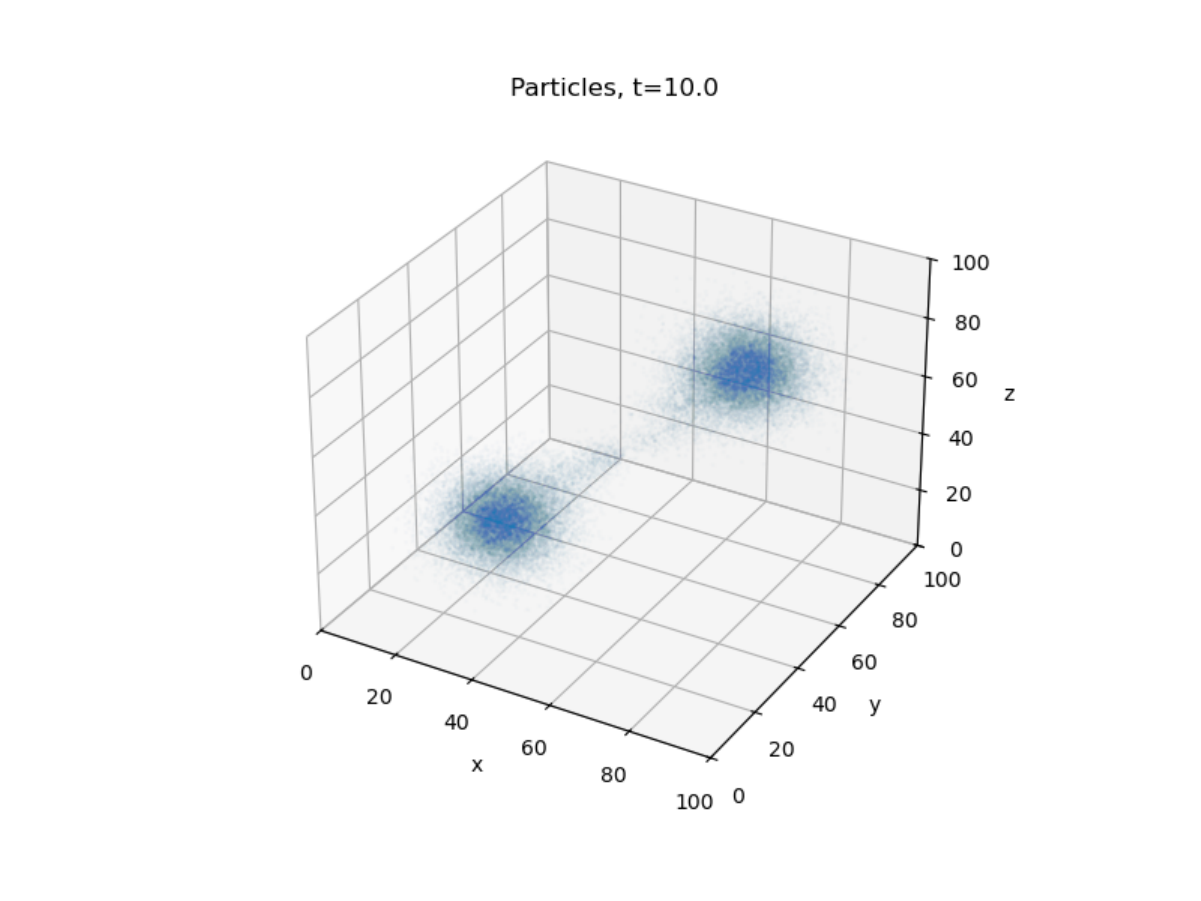}

    \end{subfigure}
    \begin{subfigure}[t]{0.48\textwidth}
        \includegraphics[width=1.1\textwidth]{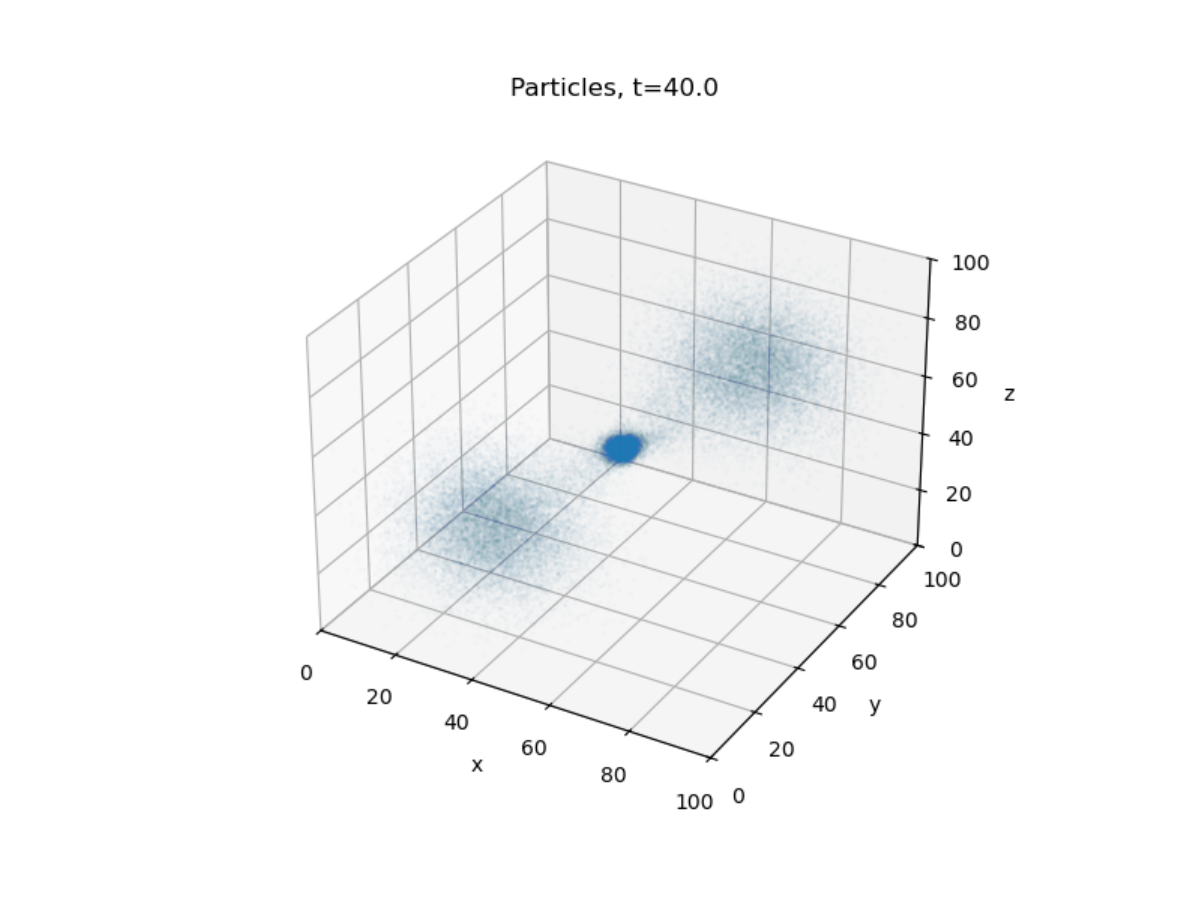}
    \end{subfigure}
    \caption[]
    {{\small Two-blob solution particles aggregating toward the center at t = 0.0, 10.0, 40.0 produced by NSIPF with \(P = 20000\)}}
    \label{2bumpparticles}
\end{figure*}

\begin{figure*}[t!]
    \centering
    \begin{subfigure}[t]{\textwidth}
        \includegraphics[width=\textwidth]{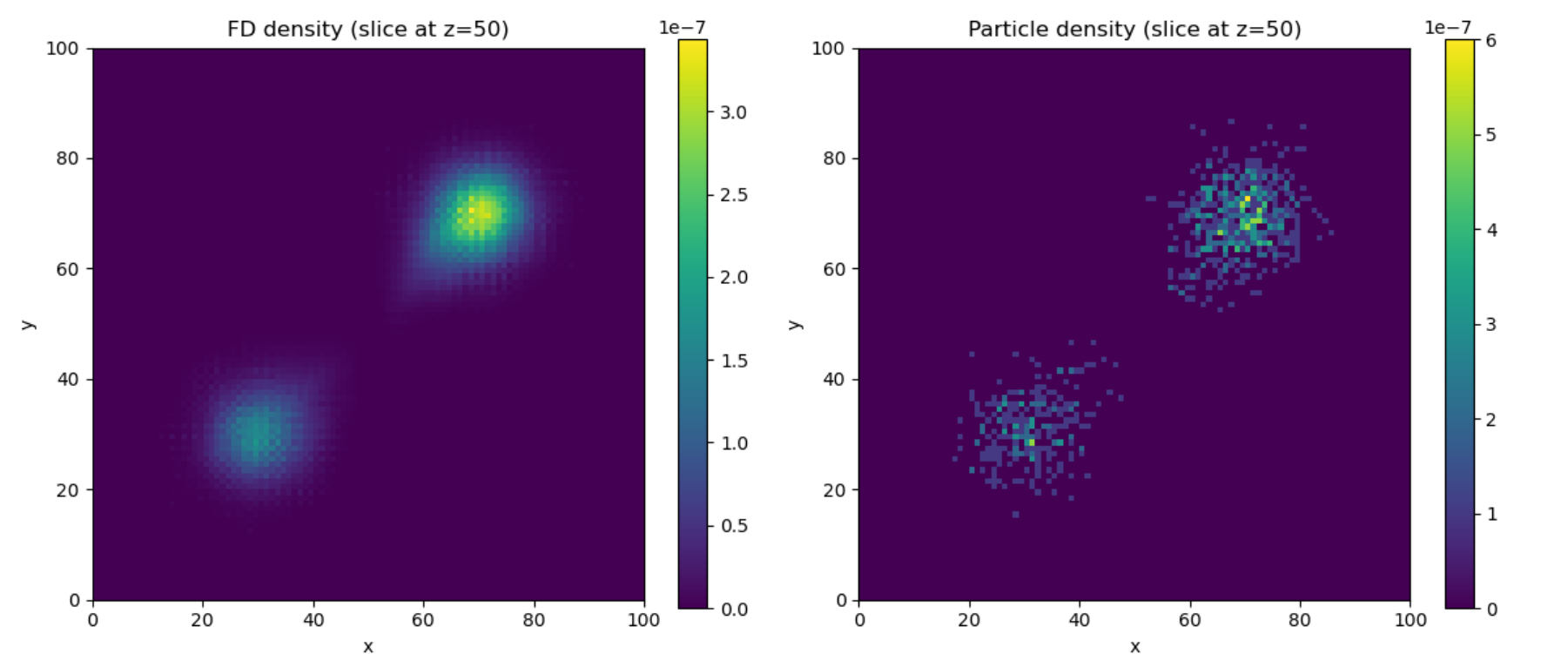}
    \end{subfigure}
    \hfill
    \begin{subfigure}[t]{\textwidth}
        \includegraphics[width=\textwidth]{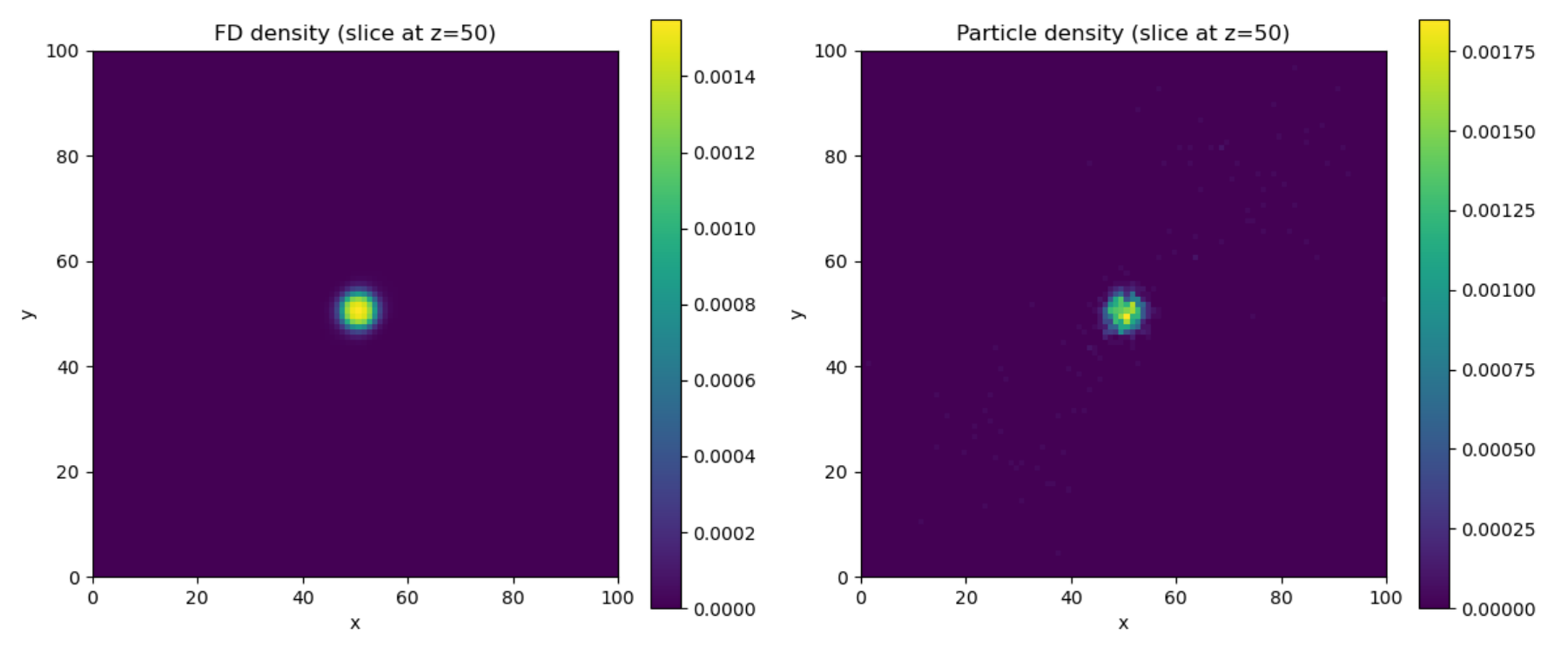}
    \end{subfigure}
    \caption[]%
    {{\small Two bump solution comparison at $t = 2.0$ and $t=40.0$ with FDM (left) and NSIPF (right) with \(P = 20000\)\footnotemark}}
    \label{2bump}
\end{figure*}

\begin{figure*}[t!]
    \centering
    \begin{subfigure}[t]{0.48\textwidth}  
        \includegraphics[width=1.1\textwidth]{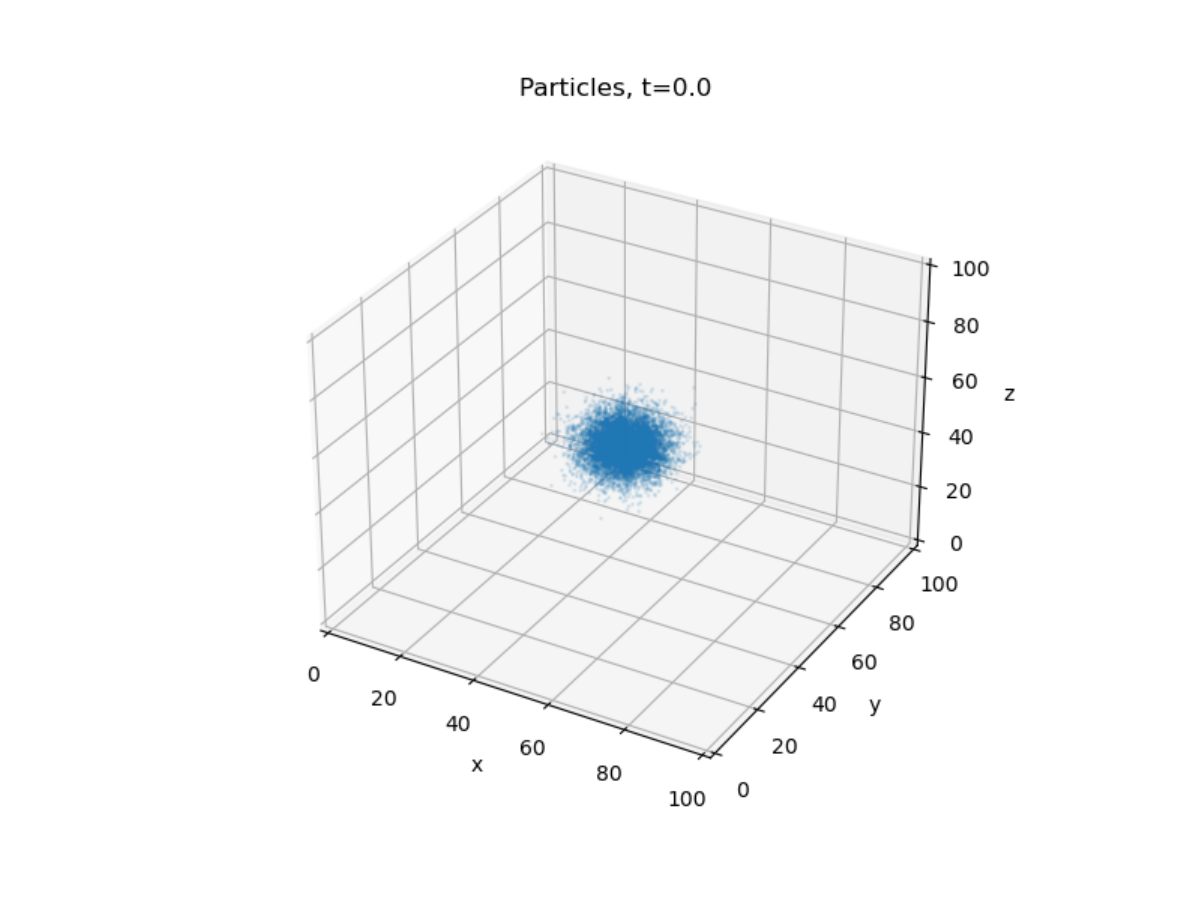}

    \end{subfigure}
    \begin{subfigure}[t]{0.48\textwidth}   
        \includegraphics[width=1.1\textwidth]{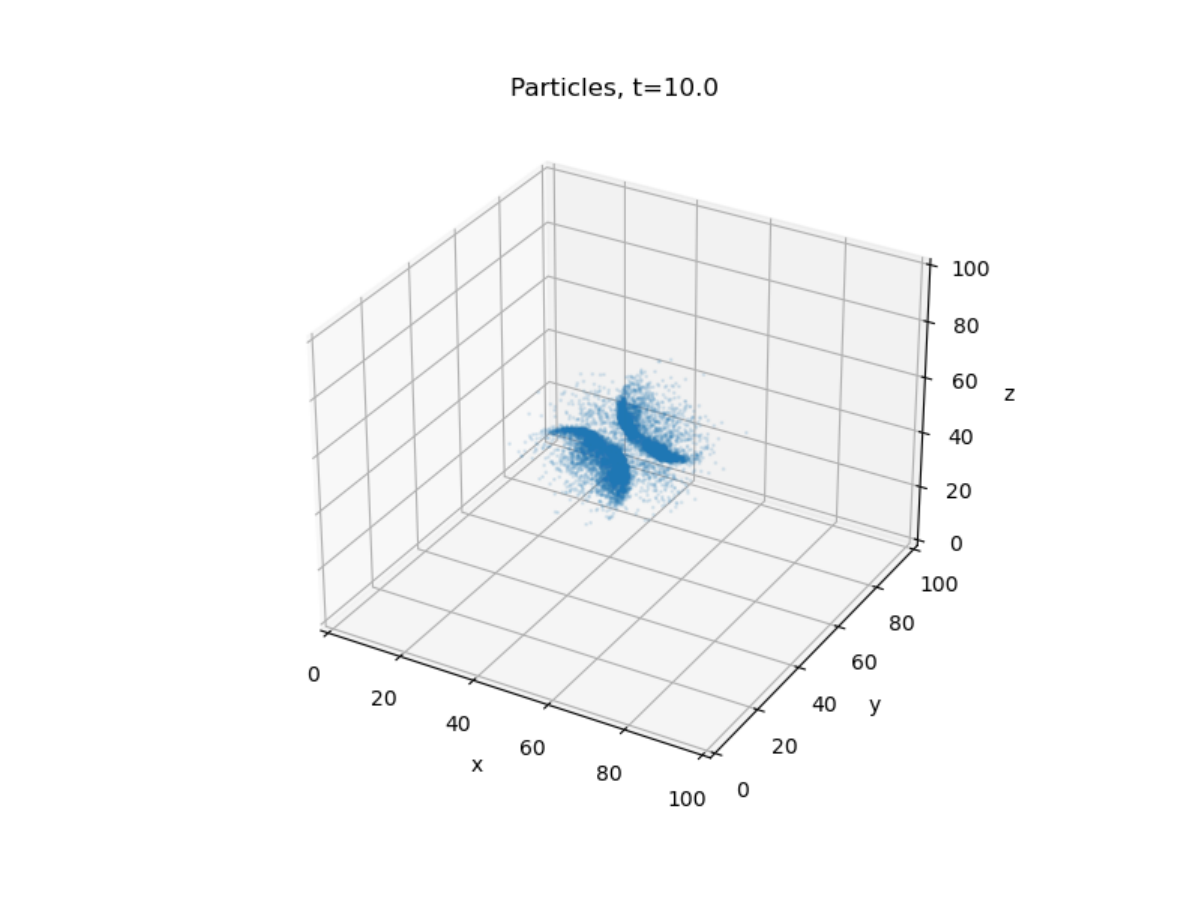}

    \end{subfigure}
    \begin{subfigure}[t]{0.48\textwidth}
        \includegraphics[width=1.1\textwidth]{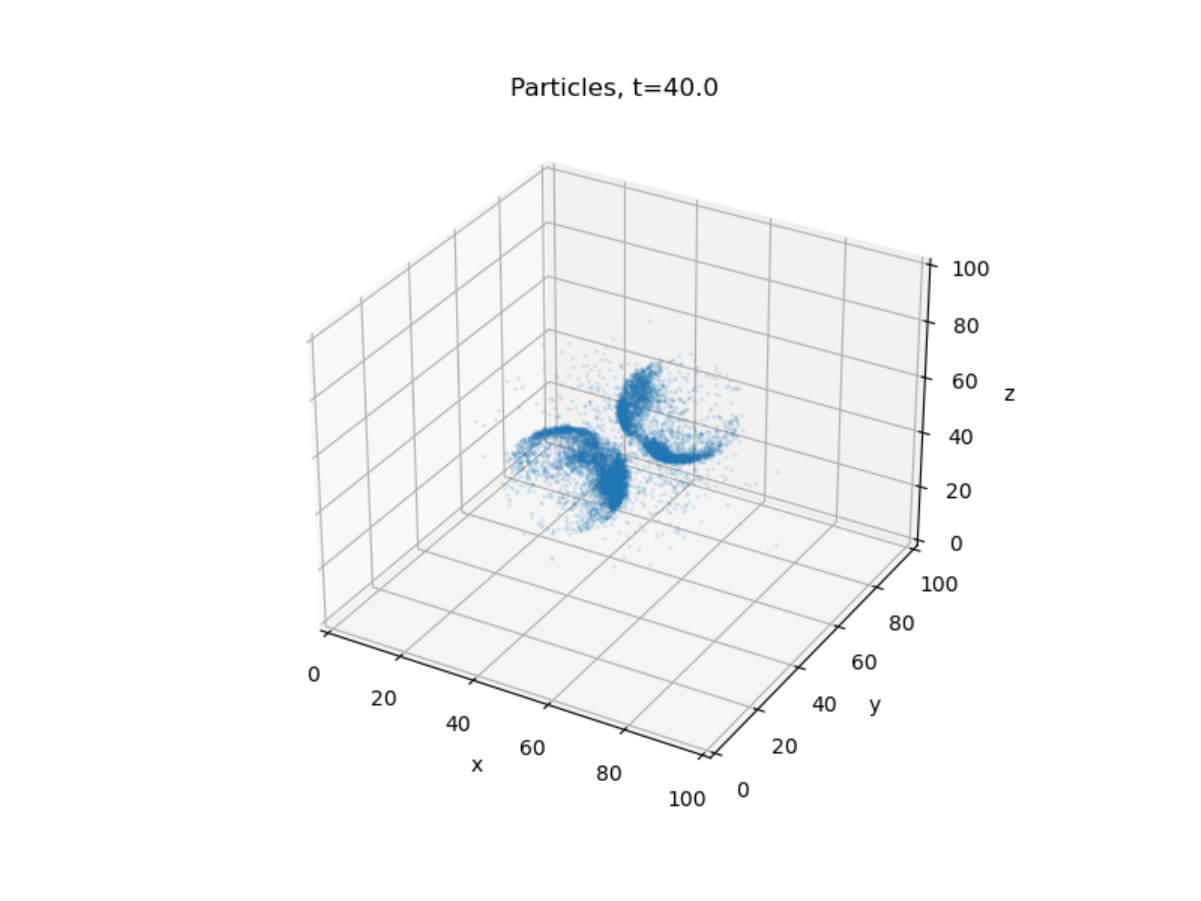}
    \end{subfigure}
    \caption[]
    {{\small One bump particles gathering around the boundary of the annuli at t = 0.0, 10.0, 40.0. produced by NSIPF with \(P = 20000\)}}
    \label{annuli}
\end{figure*}

\subsection{Comparison of numerical methods}
In this subsection, we compare the numerical methods implemented in our 3D experiments. In Table \ref{table: Comparison of Runtimes}, we compare runtimes of all three methods with timestep \(dt = 0.1\). The main baselines we compare NSIPF to are FDM and the classical SIPF. The usage of classical interpolation methods such as spline in the classical SIPF heavily extends runtimes compared to either FDM or Neural SIPF. The spline method interpolates across all particle positions in the spatial domain defined by a grid and thus results in a longer runtime. However, the neural interpolator, since it is trained on a sample of augmented solution data, takes less time when called to interpolate only the particle positions in which density or concentration is nonzero. Thus, the interpolation task is reduced from a global task of interpolating all particle positions to just a local one and a function call in our experiments and implementation. As in the annuli experiment, the neural interpolator is able to interpolate concentration data that are not blobs. \par 
We also compare classical SIPF and Neural SIPF methods with respect to particle count and see the weakness of classical methods when it comes to runtime. As we can see from Table \ref{table: classical vs. neural}, NSIPF is unburdened by the increasing particle count. We found that NSIPF runtime starts to lengthen at around 500k particles but still remains faster than traditional methods.

\begin{table}[h!]
\centering
\begin{tabular}{|c|c|c|}
\hline
    Method & Resolution & Runtime (s) \\
    \hline
    FDM & $50 \times 50 \times 50$ & 7.31\\
    & $100 \times 100 \times 100$ & 56.89\\
    & $200 \times 200 \times 200$ & 742.24\\
    \hline
    Classical SIPF (P = 20k) & $50 \times 50 \times 50$ & 2955.73\\
    & $100 \times 100 \times 100$ & 3919.37\\
    & $200 \times 200 \times 200$ & 7599.54\\
    \hline
    NSIPF (P = 20k) & $50 \times 50 \times 50$ & 8.69 \\
    & $100 \times 100 \times 100$ & 33.12 \\
    & $200 \times 200 \times 200$ & 243.86 \\
    \hline
\end{tabular}
\caption{Comparison of run times for different resolutions in 3D  (50/100/200 refers to the number of grid points/bins in each dimension).}
\label{table: Comparison of Runtimes}
\end{table}

\begin{table}[h!]
\centering
\begin{tabular}{|c|c|c|}
\hline
    P & Classical SIPF Runtime (s) & NSIPF Runtime (s)\\
    \hline
    1000 & 211.27 & 31.97\\
    5000 & 998.38 & 32.21\\
    10000 & 1959.34 & 32.42\\
    \hline
\end{tabular}
\caption{Classical vs. Neural SIPF run times with respect to particle number.}
\label{table: classical vs. neural}
\end{table}

\subsection{Numerical convergence}
In this section, we present numerical findings to convergence of the NSIPF method with respect to both the number of particles \(P\) and the timestep \(dt\). We calculate the relative \(L^2\) error, 
\begin{equation}
    \frac{\sqrt{\sum (f_{num} - f_{ref})^2}}{\sqrt{\sum f_{ref}^2}}
\end{equation}
where \(f_{num}, f_{ref}\) are the numerical and reference solutions respectively. The reference solution is taken to be an FDM output with \(dx = 0.02\) and \(dt = 0.0125\). In Figure \ref{convergence}, we compute the relative \(L^2\) error of density and concentration with respect to different particle count \(P\) and timestep \(\delta t\). Fitting the slope of the error in the log-log plot yields \(e(P) = \mathcal{O}(P^{-0.47})\) for density and \(e(\delta t) = \mathcal{O}(\delta t^{1.01})\) respectively, implying that NSIPF is of order \(-\frac{1}{2}\) with respect to \(P\) and first order with respect to timestep. \par

\begin{figure*}[t!]
    \centering
    \begin{subfigure}[t]{0.46\textwidth}  
        \includegraphics[width=1.1\textwidth]{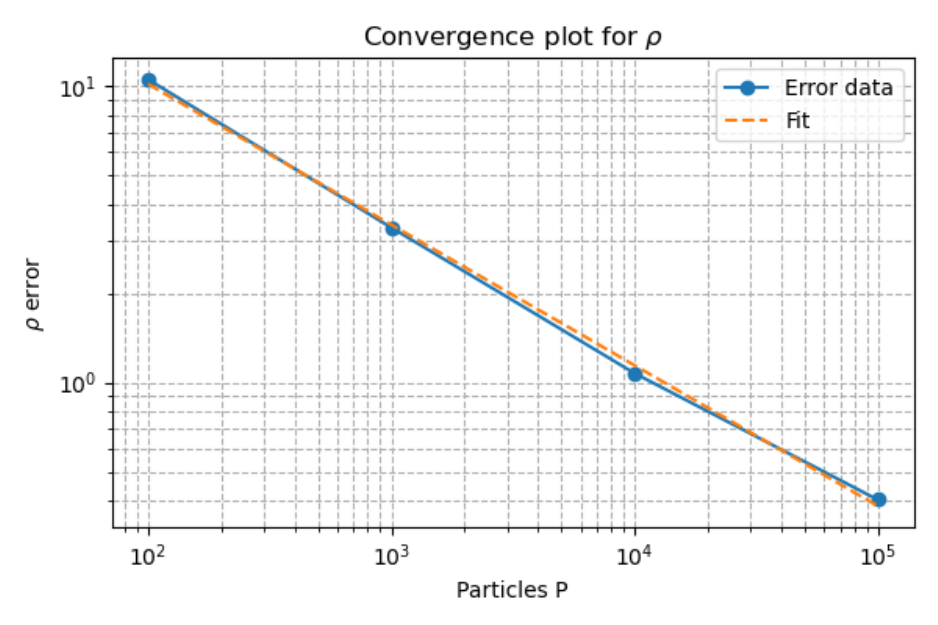}

    \end{subfigure}
    \hfill
    \begin{subfigure}[t]{0.46\textwidth}   
        \includegraphics[width=1.1\textwidth]{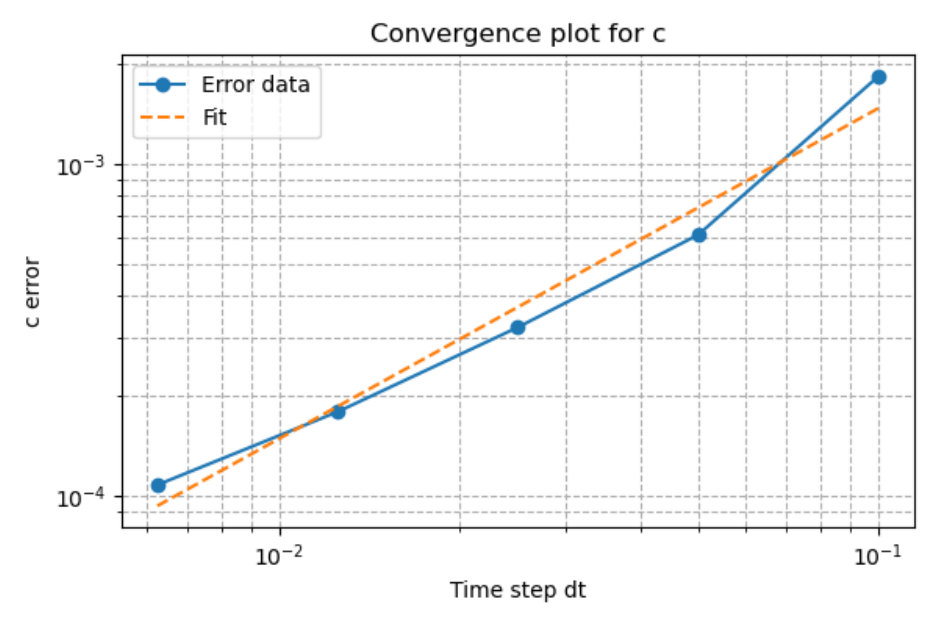}

    \end{subfigure}
    \caption[]
    {{\small Convergence plots of \(\rho, c\) using NSIPF. Fitted slopes are \(-0.47, 1.01\) respectively.}}
    \label{convergence}
\end{figure*}

\footnotetext{The resolution of the cross-section of the NSIPF solution increases when particle count increases.}

\section{Conclusion and Future Research}
We have presented an interacting particle method with neural interpolation that effectively computes solutions to the PHKS system in 3D and agrees with computations using traditional methods. Motivated by a theoretical propagation of chaos statement, we present a neural interpolator enhancing the performance of the classical SIPF. The CNN used for interpolation is trained on low cost radial solution data generated from FDM in one dimension (radial variable) and demonstrates efficacy in handling initial concentration profiles other than blobs and other parameter regimes. In future work, we aim to improve interpolation performance in runtime and speed. Finally with the particle solutions here as training data, we aim to develop a physics-aware  (mass and non-negativity preserving) generative AI model for predicting cell dynamics, 
bypassing the 
reliance on mechanistic models entirely.

\section{Acknowledgments}
This work was partially supported by NSF grant DMS-2309520, by the Swedish Research Council grant 
no.2021-06594 at the Institut Mittag-Leffler in Djursholm, Sweden, as well as the E. Schr\"odinger Institute in Vienna, Austria, both during JX's visit in the Fall of 2025.
 \medskip

The authors would like to thank Dr. Antoine Diez for helpful discussions on propagation of chaos for the PHKS system.

\bibliographystyle{siam}
\bibliography{refs}

\end{document}